\documentclass[11pt,reqno]{amsart}
\usepackage{amscd,amssymb,amsmath,amsthm}
\usepackage{mathtools}
\usepackage[english]{babel}
\usepackage[T1]{fontenc}
\usepackage[arrow,matrix]{xy}
\usepackage{graphicx,subfigure,tikz,tikz-3dplot}

\usetikzlibrary{positioning,matrix,arrows,calc}
\usepackage{cite}
\usepackage{dsfont}
\usepackage{enumitem}
\usepackage{bbm}%fuer die dicke Indikatorfunktionseins -> \mathbbm
\usepackage{hyperref}

\hypersetup{
	colorlinks=true,
	linkcolor=blue}
\usepackage{array}
\usepackage{xargs}                      % Use more than one optional parameter in a new commands
\usepackage[colorinlistoftodos,prependcaption,textsize=tiny]{todonotes}

\oddsidemargin=0.3in \evensidemargin=0.3in \theoremstyle{plain}
\newtheorem{theorem}{Theorem}[section]
\newtheorem{lemma}[theorem]{Lemma}
\newtheorem{proposition}[theorem]{Proposition}
\newtheorem{corollary}[theorem]{Corollary}

\theoremstyle{definition}

\newtheorem{definition}[theorem]{Definition}

\numberwithin{equation}{section} 

\newcommand{\Prob}{\mathbb{P}}

\newcounter{hypocounter}
\renewcommand\thehypocounter{(H\arabic{hypocounter})} % this produces '(H1)'
\title[Analytic Dependence for Random Matrix Products]{Analytic Dependence of the Lyapunov moment function and the projective stationary measure for Random Matrix Products}
\author{C. Chalhoub, V.P.H. Goverse, J.S.W. Lamb, M. Rasmussen}

\address{C. Chalhoub\\  Department of Mathematics,
	Imperial College London,
	180 Queen's Gate,
	London SW7 2AZ,
	United Kingdom}
\email{christopher.chalhoub21@imperial.ac.uk}

\address{V.P.H. Goverse\\  Department of Mathematics,
	Imperial College London,
	180 Queen's Gate,
	London SW7 2AZ,
	United Kingdom}
\email{vincent.goverse21@imperial.ac.uk}

\address{J.S.W. Lamb\\  Department of Mathematics,
	Imperial College London,
	180 Queen's Gate,
	London SW7 2AZ,
	United Kingdom\newline
	International Research Center for Neurointelligence (IRCN), The University of Tokyo, Tokyo 
113-0033, Japan\newline
Centre for Applied Mathematics and Bioinformatics, Department of Mathematics and Natural Sciences, Gulf University for Science and Technology, Halwally, 32093 Kuwait.
	}
\email{jsw.lamb@imperial.ac.uk}

\address{M. Rasmussen\\  Department of Mathematics,
	Imperial College London,
	180 Queen's Gate,
	London SW7 2AZ,
	United Kingdom}
\email{m.rasmussen@imperial.ac.uk}	
\begin{document}

\begin{abstract}
% We consider the product of i.i.d. random matrices sampled according to a probability measure $\mu$ on a compact subset of $GL(d,\mathbb{R})$ and establish analyticity of the Lyapunov moments and the stationary measure in the measure $\mu$.
We consider the product of i.i.d. random matrices sampled according to a probability measure $\mu$ supported on a strongly irreducible and proximal subset of a compact set $S\subset GL(d,\mathbb{R})$. We establish the local analyticity of the Lyapunov moment function and the unique stationary measure on the projective space with respect to $\mu$ in the total variation topology. As a consequence, we obtain the analyticity of the asymptotic variance and all higher-order Lyapunov moments.
   % Here is the philosophy (REWRITE) : we want to show that the growth rate of the expected number of escapes from a ball as the system approaches chaos is 1/lambda. To do that properly, we work on families of measures on matrices that induce systems whose lyapunov exponents approach zero from below. In this context we obtain upper and lower bounds on expected number of escapes of the form C/lambda where the constants C do not depend on the family of measures! So this is strong: the growth rate only depends on the actual measure through the lyapunov exponent.
\end{abstract}
\maketitle

\section{Introduction}

% A discrete-time random dynamical system typically refers to an iteration of maps driven by noise. An important phenomenon that emerges in such systems exhibiting contractive behaviour is synchronization: the convergence of trajectories starting from different initial conditions under a common noise realization. Some classical results in synchronization include those of Le Jan \cite{LeJ87} and Baxendale \cite{Bax91} where the author has established the existence and uniqueness of an invariant measure on the diagonal for the two-point motion in a negative top Lyapunov exponent $\lambda$ regime in the context of stochastic differential equations under mild conditions.  \\

% Such limit theorems often mask important transient phenomena. Simulations show that trajectories don't converge uniformly in the noise.

Products of random matrices form a classical topic in probability theory, ergodic theory, and the study of random dynamical systems. Since the foundational work of Furstenberg and Kesten \cite{furstenberg1960products}, they have served as a fundamental model for multiplicative stochastic processes. Given i.i.d. matrices $A_n\in GL(d,\mathbb{R})$ with common law $\mu$, the asymptotic growth of $A_\omega^n = A_n \dots A_1$ is governed by the top Lyapunov exponent \[ \lambda_\mu := \lim_{n\to \infty} \frac{1}{n} \mathbb{E}_\mu [\ln \|A_\omega^n\|] \]
where $\mathbb{E}_{\mu}$ is the expectation with respect to $\Prob_\mu \equiv \mu^{\otimes \mathbb{N}}$. \\

A major theme in the theory concerns the dependence of Lyapunov exponents on the underlying probability measure. Under some irreducibility conditions, Furstenberg and Kifer \cite{furstenbergkifer1983} established the continuity of $\mu \mapsto \lambda_\mu$ in a specific weak topology on the space of probability measures. This result has been recently improved by Avila and Viana \cite{avila2023continuity}, who established continuity in a topology corresponding to weak$^*$-closeness of the distributions and Hausdorff-closeness of their supports.

In the case of positive matrices, Ruelle has proved analytic dependence on the matrix entries \cite{MR534172}.
Other works have considered the question of continuity when
the matrices depend smoothly on a finite-dimensional parameter \cite{hennion1984loi,kifer1982perturbations}. Spectral gap techniques for the associated transfer operator, considered by Le Page \cite{lepage1989} and Guivarc’h and Raugi \cite{guivarc1985frontiere}, were central to the study of regularity properties of $\lambda_\mu$, and the authors provided settings where the Lyapunov exponent is differentiable in the finite-dimensional parameter. In the case of a finite number of invertible matrices, Peres \cite{peres2006analytic,MR1158741} proved real-analyticity of the Lyapunov exponent on the open simplex of positive probability vectors. Our work is  motivated by a recent result by Amorim, Durães and Melo \cite{amorim2025analiticity}, where the authors use tools from complex analysis on Banach spaces to prove the analyticity of $\lambda_\mu$ in $\mu$ with respect to the total variation norm. To extend their work, we use operator perturbation arguments (see Kato \cite{MR203473}) applied to the twisted Koopman operator rather than establishing the analytic variation of observables on the projective space directly.\\

Since the proof of existence and uniqueness of the Furstenberg stationary measure (\cite{furstenberg1960products}) on the projective space under strong irreducibility and proximality conditions, a series of results began to clarify its finer properties including Hölder regularity and convergence in direction. Later results achieved a classification of all stationary probability measures on the projective space for i.i.d. random matrix products with finite first moment. Notable contributions include \cite{guivarc1985frontiere}, \cite{AounGuivarc}, \cite{aounsert1} and \cite{aounsert2}.\\  

The review above suggests that the stability properties of the top Lyapunov exponent are well-understood. However, little is known regarding other important observables of the system, for instance the asymptotic variance \[ \sigma^2_\mu \coloneqq \lim_{n\to \infty} \frac{1}{n} \mathbb{E}_\mu [ (\ln \Vert A_\omega^n\Vert  - \lambda_\mu n)^2 ]. \]
The Lyapunov moment function $\Lambda_\mu(q)$ defined below captures the asymptotic information of the system. For instance, $\Lambda_\mu '(0)=\lambda_\mu$ and $\Lambda_\mu ''(0)=\sigma^2_\mu$. Our main contribution is establishing analyticity of $\Lambda_\mu(q)$ in $\mu$ with respect to the total variation norm. In particular, we prove that all the Lyapunov moments $\Lambda^{(k)}_\mu(0)$ are analytic in $\mu$, which extends the result of \cite{amorim2025analiticity} to all the higher moments of the Lyapunov moment function, while also providing an alternative proof of the result of \cite{amorim2025analiticity}. We also show that the stationary probability measure (seen as a bounded linear functional on the projective space) varies analytically in $\mu$. \\

The results of this paper clarify the picture of the analytic structure of random matrix products and open avenues for further investigation of infinite-dimensional perturbations beyond the linear context for more general random dynamical systems. Indeed many open questions remain on the dependence of invariant manifolds and invariant measures on the underlying noise distribution in that context.

% {\color{red} TO DO INTRO:
% \begin{itemize}
%     \item Intro: this part could be improved 
%     \item CLEAN UP REFERENCES (Chris)
%     \item add citation to Cargi Sert global existence of lyapunov moment function (Its very good; annals of prob) (I have done please check if you agree)
%     \item Check acknowledgements
% \end{itemize}

% }

\section{Set-up and Statement of the Main Results}
Denote by $\mathbb R^d$ the $d$-dimensional Euclidean space and $\mathbb S^{d-1}$ the $d-1$-dimensional unit sphere. For $x\in \mathbb{R}^d$, let $\Vert x \Vert = \left( \sum_{i=1}^d |x_i|^2 \right)^{1/2}$ be the Euclidean norm. Let $GL(d,\mathbb{R})$ denote the space of $d \times d$ invertible matrices with real-valued entries endowed with the matrix operator norm \[ \Vert A \Vert = \sup_{x\in \mathbb{S}^{d-1}} \Vert Ax\Vert.  \] %We follow the presentation in \cite{RM_Book}

We endow $GL(d,\mathbb{R})$ with the Borel sigma-algebra. Let $S\subset GL(d,\mathbb{R})$ and fix a probability measure $\mu$ on $S$. Let $\Omega = S^{\mathbb{N}}$ and $\mathcal{F}$ the sigma-algebra generated by the cylinder sets. Let $\Prob_\mu = \mu^{\otimes \mathbb{N}}$ and $\mathbb{E}_\mu$ the expectation with respect to $\Prob_\mu$. Each $\omega\in \Omega$ is a sequence of independent matrices $\omega = (\omega_1, \omega_2,\dots)$ sampled according to $\mu$.\\

We define the product of random matrices on $(\Omega, \mathcal{F},\Prob)$ at time $n\in \mathbb{N}$ as the random variable $A_\omega^n: \Omega\to GL(d,\mathbb{R})$ given by $A_\omega^n = \omega_n \omega_{n-1} \dots \omega_1 $. For $x\in \mathbb{R}^d$, $(A_\omega^nx)_{n\geq0}$ is a Markov process on $\mathbb{R}^d$.

We define certain characteristics we impose on the support $W\subset GL(d,\mathbb{R})$ of the measure $\mu$ from \cite{RM_Book}.

\begin{definition}
    The set $W$ is said to be \textit{proximal} if there exists $n\geq1$ and a tuple $(A_1,\dots,A_n)\in W^n$ such that the product $\prod_{i=1}^n A_i$ admits an algebraic simple dominant eigenvalue.
\end{definition}

\begin{definition}
    The set $W$ is said to be \textit{strongly irreducible} if there does not exist a finite family of linear subspaces $ \{0\} \subsetneq V_1,\dots, V_k \subsetneq \mathbb{R}^d$ whose union is preserved by every element of $W$, i.e. such that \[ A(V_1 \cup V_2 \dots \cup V_k ) = V_1 \cup V_2 \dots \cup V_k \]
    for all $A\in W$.
\end{definition}

A probability measure is said to be proximal and strongly irreducible if its support satisfies these conditions. We would like to consider the dynamics of the system on the projective space of directions $\mathbf{P}(\mathbb{R}^d)$ which we associate to the unit sphere $\mathbb{S}^{d-1}$. We define the induced process $(\theta_n)$ on $\mathbb{S}^{d-1}$ by $\theta_0 = \frac{x}{\Vert x\Vert }$ and $\theta_n = \frac{A_\omega^n(x)}{\Vert A_\omega^n(x)\Vert }$. Note that $(\theta_n)$ is a Markov process unlike the process $(\Vert A_\omega^n\Vert )_n$ that describes the norm of the random matrix product.

A probability measure $\nu$ on $\mathbb{S}^{d-1}$ is said to be \textit{$\mu$-invariant} if for every bounded measurable function $f$ on $\mathbb{S}^{d-1}$, we have \[ \int_{\mathbb{S}^{d-1}} f (x) \ d\nu(x) = \int_S \int_{\mathbb{S}^{d-1}} f \left( \frac{A x}{\Vert A x\Vert }  \right) \ d\nu(x) d\mu(A). \]

We define the Lyapunov exponents of the system. The top Lyapunov exponent yields the asymptotic exponential growth rate of the norm of the product.

\begin{definition}
    For a matrix $M\in GL(d,\mathbb{R})$, denote by $\tilde{\sigma}_1(M)\geq \tilde{\sigma}_2(M)\geq \dots \geq\tilde{\sigma}_d(M)$ its singular values. Let $(A_\omega^n)$ be a product of random matrices satisfying $ \int_S \ln^+ \Vert A\Vert d\mu(A)<\infty$. The associated Lyapunov exponents $(\lambda_p)_{p=1}^d$ are defined by 
    \[ \lambda_p \coloneqq \lim_{n\to \infty} \frac{1}{n} \mathbb{E}[\ln \tilde{\sigma}(A_\omega^n) ] . \]
    In particular, the top Lyapunov exponent is given by 
    \[\lambda \coloneqq \lambda_1 = \lim_{n\to \infty} \frac{1}{n} \mathbb{E} [\ln \Vert   A_\omega^n \Vert ] .\]
\end{definition}

Furstenberg and Kesten (\cite{furstenberg1960products}) proved the following fundamental result.

\begin{theorem}{\cite{furstenberg1960products}}
    Let $\mu$ be an irreducible probability measure on $S\subset GL(d,\mathbb{R})$ such that $ \int_S \ln^+ \Vert A\Vert d\mu(A) <\infty$. Then for all $x\in \mathbb{S}^{d-1}$, we have \[ \Prob_\mu \left[  \lambda = \lim_{n\to \infty }\frac{1}{n} \ln \Vert A_\omega^n x\Vert   \right]=1.  \]
\end{theorem}

We would like to work in a setup where $\lambda\ne \lambda_2$ as this will impose strong mixing properties on the process $(\theta_n)$. We state this formally in the following theorem due to Guivarc'h and Raugi.

\begin{theorem}{\cite{guivarc1985frontiere}}
    Suppose that the probability measure $\mu$ is contracting and strongly irreducible. Then $\lambda>\lambda_2$ and there exists a unique $\mu$-invariant probability measure $\nu$ on $\mathbb{S}^{d-1}$.
    \end{theorem}

In addition to the assumptions of strong irreducibility and proximality, we also impose that the supports of the measures $\mu$ are contained in a  compact subset $S$ of $GL(d,\mathbb{R})$ so the matrix norms are uniformly bounded, i.e. \[ \eta \coloneqq \sup_{A\in S} \max \{ \Vert A\Vert , \Vert A^{-1}\Vert  \} <\infty. \]
% \begin{theorem}
%     For a system of random matrix products
%     Contracting + Strong irreducibility -> Unique stat measure on $\mathbb S^{n-1}$
%     (Spectral gap in $\mathcal \mathcal{C}^\alpha$)
% \end{theorem}

The asymptotic variance of the random matrix product $(A_\omega^n x)$ where the matrices are sampled from a measure $\mu$ is
\[  \sigma^2 \coloneqq \lim_{n\to \infty} \frac{1}{n} \mathbb{E} [ (\ln \Vert A_\omega^n\Vert  - \lambda n)^2 ].  \]
By the central limit theorem for random matrix products , see \cite[Theorem V.5.1]{RM_Book}, since $S$ is proximal, strongly irreducible and compact (more generally the matrices satisfy finite exponential moments), we have $\sigma^2>0$.

We also define the Lyapunov Moment function (see \cite{ArnoldStabFormula,baxendale1998stability})) of the random matrix product:
\begin{equation}\label{momlf}
	\Lambda(q) \coloneqq \lim_{n\to\infty} \frac{1}{n}\ln  \left(\mathbb{E} \left[ e^{q \ln \Vert A_\omega^n v\Vert} \right]\right), \quad\text{for } v\in \mathbb S^{d-1}, q\in \mathbb{C} % \label{eq:Lyapunovmoment}
\end{equation}
where the limit does not depend on $v$. Global existence has been recently shown in \cite{MR3945748}. Computations (see Lemma \ref{LyapProps} below) show that $\lambda = \Lambda'(0)$ and $\sigma^2 = \Lambda''(0)$. In general, we refer to $\Lambda^{(k)}(0)$ (the $k^{th}$ derivative at zero) as the $k^{th}$ Lyapunov moment of the random matrix product.\\

% \textbf{Family of probability measures on $GL_d(\mathbb R)$}
% \subsection{Ball of probability measures}
Denote by $P(S)$ the metric space of probability measures on $S\subset GL(d,\mathbb{R})$ under the total variation distance \[ \Vert \mu_1-\mu_2 \Vert_{TV} = 2 \sup_{B\subset S} |\mu_1(B)-\mu_2(B)| = \sup_{f\in B_b(S): ||f||_\infty =1} \left| \int_S f(x) d\mu_1(x) - \int_S f(x)d\mu_2(x) \right| \]
where $B_b(S)$ denotes the set of bounded measurable functions on $S$.
We will also consider the Banach space of finite complex measures $M(S)$. In this context, the total variation norm is defined as 
\[ \| \mu \|_{TV} := |\mu|(S) = \sup_{f\in B_b(S): \| f\|_\infty\leq 1} \left| \int_S f(x)d\mu(x)  \right|.  \]

For a measure $\mu\in M(S)$, we write $\lambda_{\mu}$, $\Prob_\mu$, $\mathbb{E}_\mu$, $\nu_\mu$, $\Lambda_\mu(q)$ and $\sigma^2_\mu$ for the induced Lyapunov exponent, product measure, expectation, unique stationary measure, asymptotic variance and Lyapunov moment function respectively.\\

We can now state the main result of the paper. We prove that the Lyapunov moments of the random matrix product and the stationary probability measure on $\mathbb{S}^{d-1}$ vary analytically in the measure $\mu$.

\begin{theorem}\label{MainTheorem1Analyticity}
    Let $S\subset GL(d,\mathbb{R})$ be compact, and let $\mu_0\in P(S)$ such that the support of $\mu_0$ is strongly irreducible and proximal. Then, there exist open neighbourhoods $U_\mu\subset M(S)$ of $\mu_0$ and $U_q\subset\mathbb{C}$ of $0$ such that:
    \begin{enumerate}
        \item The map \[ (q,\mu) \mapsto \Lambda_\mu(q) \]
    is analytic on $U_q\times U_\mu$.
    \item For all $k\in \mathbb{N}$, the map \[ \mu\mapsto \Lambda^{(k)}_\mu (0) \]
    is analytic on $U_\mu$. In particular, $\lambda_{\mu}=\Lambda'_\mu(0)$ and $\sigma^2_\mu=\Lambda''_\mu(0)$ are analytic in $\mu$.
    \item The map \[ \mu \mapsto \nu_\mu \] is analytic on $U_\mu$.
    \end{enumerate}
\end{theorem}

We explain the notion of analyticity used in Theorem \ref{MainTheorem1Analyticity}. We prove Fréchet holomorphicity (definition in Section \ref{Prelims_Section}) of these maps defined on open subsets of the Banach spaces $\mathbb{C}\times M(S)$ for (1) and $M(S)$ for (2) and (3). The set $M_0(S) = \{\mu\in M(S): \mu(S)=0\}$ is a closed linear subspace of $M(S)$, hence also a Banach space. We can then transfer the holomorphic structure of $M_0(S)$ to the affine subspace of complex measures with unit total mass \[M_1(S) = M_0(S)+\mu_0 = \{\mu\in M(S): \mu(S)=1\}\]
since $\mu_0 \in P(S)$ (see \cite[Section 2.1]{amorim2025analiticity}).\\

In the case of (3), we prove Fréchet holomorphicity of the function $\nu: M(S)\to (\mathcal{C}^\alpha (\mathbb{S}^{d-1}))^*$ (the dual space of the Banach space of $\alpha-$Hölder continuous functions on the unit sphere) that sends a measure $\mu$ to the left eigenvector of the transfer operator $\mathcal{P}_{0,\mu}$ defined in Section \ref{Prelims_Section} for the eigenvalue $1$, i.e. $\nu_\mu \mathcal{P}_{0,\mu} = \nu_\mu $. Indeed the linear functional $\nu_\mu$ need not be a measure, it only represents the unique stationary measure on the projective space when $\mu$ is a strongly irreducible and proximal probability measure.\\

We outline the main ideas of the proof. We first start by presenting some tools we need from the theory linear operators on Banach spaces in Section \ref{Prelims_Section}. Analyticity of the transfer operator is established in Section \ref{Analyticity_of_Koopman_Section} and followed by perturbation theory arguments needed for the proof of Theorem \ref{MainTheorem1Analyticity} in Section \ref{Proof_Of_Main_Section}.

\section{Preliminaries}\label{Prelims_Section}
For $\alpha \in (0,1)$, let $\mathcal{C}^\alpha(\mathbb{S}^{d-1})$ denote the Banach space of $\alpha$-Hölder continuous functions $f: \mathbb{S}^{d-1}\to \mathbb{C}$ endowed with the norm $ \Vert f\Vert _{\mathcal{C}^\alpha} = \Vert f\Vert _\infty + [f]_\alpha $ where 
\[ [f]_\alpha = \sup_{x\ne y} \frac{|f(x)-f(y)|}{\Vert x-y\Vert ^\alpha}. \]
% Denote $\mathbbm{1}_{\mathbb{S}^{d-1}}:\mathbb{S}^{d-1}\to \mathbb{C}$ as the constant function equal to 1 on $\mathbb{S}^{d-1}$. 
Let $\mathcal{L}(\mathcal{C}^\alpha(\mathbb S^{d-1})$ denote the space of bounded linear operators from $\mathcal C^\alpha(\mathbb S^{d-1})$ to itself and denote by $(\mathcal{C}^\alpha(\mathbb S^{d-1}))^*$ its dual space. 
For $q\in \mathbb{C}$ and $\mu\in P(S)$, one of the operators of interest is the \textit{$q$-twisted Koopman operator} $\mathcal{P}_{q,\mu} \in \mathcal{L}( \mathcal{C}^\alpha(\mathbb{S}^{d-1}) )$ defined by \[ \mathcal{P}_{q,\mu} f (v) = \int_{S}   e^{q \ln \Vert A v\Vert} \  f \left( \frac{A v}{\Vert A v\Vert }  \right)  d\mu(A), \ \ \ \text{for } f\in \mathcal{C}^\alpha(\mathbb{S}^{d-1}), v\in \mathbb{S}^{d-1}. \]

For $f\in C^\alpha(\mathbb{S}^{d-1})$ and $\nu$ a measure in the dual space, we set $\langle f,\nu\rangle=\int_{\mathbb{S}^{d-1}} f \  d\nu$. \\

The result below yields the fundamental connection between the Lyapunov moment function and the twisted Koopman operator. It also establishes a spectral gap for the operator, an essential result for transfering analyticity from the operator to the eigenvalues.

\begin{theorem}\cite[Theorem V.4.3]{RM_Book}\label{Spectral_Gap_Operator}
    Let $S\subset GL(d,\mathbb{R})$ and $\mu\in P(S)$ be a strongly irreducible and proximal measure. There exist $\alpha_0>0$ such that for all $\alpha<\alpha_0$, there exists $q_0>0$ such that for $|q|<q_0$, $\mathcal{P}_{q,\mu}:\mathcal{C}^\alpha(\mathbb S^{d-1})\to \mathcal{C}^\alpha(\mathbb S^{d-1})$ is a bounded linear operator with an isolated simple eigenvalue $e^{\Lambda_\mu(q)}$. 
    % Furthermore, there exist four maps $r\colon B(0,q_0)\to \mathbb{C},$ $\phi_\cdot\colon B(0,q_0)\to \mathcal{C}^\alpha(\mathbb S^{d-1})$, $\nu_\cdot: B(0,q_0)\to (\mathcal{C}^\alpha(\mathbb S^{d-1}))^*$ and $Q_\cdot\colon B(0,q_0)\to L(\mathcal{C}^\alpha(\mathbb S^{d-1})),$ such that
    % \begin{equation}\label{eq:decomposition}
    %     \mathcal{P}_q f = r(q)\langle\nu_q,f \rangle \phi_q + Q_q f \ \ \ \ \ \text{for all } f \in \mathcal C^\alpha(\mathbb S^{d-1})
    % \end{equation}
    % where $Q_q\phi_q = 0,$ $\langle\nu_q,Q_q\phi\rangle = 0$, $\langle \nu_q, \phi_q\rangle = 1$ and 
    % \[\limsup_{n\to\infty} \Vert Q_q^n \Vert^{1/n} < e^{\Lambda(-q)} .\]
    % For $q = 0, $ we have $r(0) = 1, \kappa_0 = \nu$ (the invariant measure) and $\phi_0 = \mathbbm 1_{\mathbb S^{d-1}}$.
\end{theorem}

The goal is to establish analyticity of the eigenvalues and eigenfunctions of $\mathcal{P}_{q,\mu}$ in $q$ and $\mu$ in a neighbourhood of $(0,\mu_0)$. Since the dominant eigenvalue of this operator at $\mu_0\in P(S)$ is $e^{\Lambda_\mu(q)}$, immediate consequences of this result (Corollary \ref{Corr_analyticity} and Corollary \ref{corr_statio}) yield the proof of Theorem \ref{MainTheorem1Analyticity}.\\

We state some properties of the Lyapunov moment function.

\begin{lemma}[\cite{RM_Book}]\label{LyapProps}
    Let $S\subset GL(d,\mathbb{R})$ be compact and $\mu\in P(S)$ be a strongly irreducible and proximal measure. Then
    \begin{enumerate}[]
        \item $\Lambda_\mu(q)$ does not depend on $v\in \mathbb{S}^{d-1}$ for $q<|q_0|$,
        \item $\Lambda_\mu(q) \geq \lambda_{\mu} q$ for $q<|q_0|$,
        \item $\Lambda_\mu (0)=0 $, $  \Lambda'_\mu (0) = \lambda_{\mu}  $ and $ \Lambda_\mu''(0)=\sigma_\mu^2$,
        \item $\Lambda_\mu(q)$ is convex and analytic on $B(0,q_0)$,
     %   \item $e^{\Lambda (p)} = \sigma(\mathcal{P}_p)$ for all $p\in B(0,q_0)$, where the RHS denotes the spectral radius of the p-twisted Koopman operator.
    \end{enumerate}
\end{lemma}

% By computing the first and second derivative of the Lyapunov moment function (simple computations){\color{These have to be there. I think they are not trivial...}}, we can derive the following asymptotics for the moments of the log norm:
% \begin{equation} \label{log2}
%     \mathbb{E}[  (\ln \Vert A_\omega^nx\Vert )^2 ] \sim \sigma^2 n + \lambda^2 n^2,
% \end{equation}  
% and
% \begin{equation}\label{log3}
%     \mathbb{E}[  (\ln \Vert A_\omega^nx\Vert )^3 ] \sim n \rho + 3 \lambda\sigma n^2 + \lambda^3 n^3 
% \end{equation} 
% where $\rho = \frac{d^3}{dp^3} \Lambda(0)$.
% {\color{red} the notation is very imprecise... Maybe lets make rigourous in a lemma}

% \begin{proof}
%     \todo[inline]{To be done, do we need to write it out though? Also, RM Book doesnt explicitly state that $phi_q$ analytic, this is a consequence of the spectrum disk thing and Kato, so maybe state separately.}
% \end{proof}

where $B(0,q_0)$ denotes the open ball of radius $q_0$ in $\mathbb{C}$. Also note that, by \cite[Section~2.1]{LDP1}, the operator $\mathcal{P}_{q,\mu}$ has a unique probability eigenmeasure $\nu_{q,\mu}$ on $\mathbb{S}^{d-1}$ (the eigenvector of the dual operator $P^*_{q,\mu}$) corresponding to the eigenvalue $e^{\Lambda_\mu(-q)}$. At $q=0$ this coincides with the Furstenberg stationary measure $\nu_\mu$. It need not be a probability measure when $\mu\in M(S)$ is not a probability measure. From now on, when discussing the (dominant) eigenfunction of the Koopman operator, we refer to the unique normalised eigenfunction, with $\langle \nu_{q,\mu}, \phi_{q,\mu}\rangle =1 .$\\

We state a technical lemma allowing us to control the $\mathcal{C}^\alpha$ norm of the operator. Recall that $\eta = \sup_{A\in S} \max\{ \|A\|,\|A^{-1}\| \} $.

\begin{lemma}\cite[Lemma V.4.2]{RM_Book}\label{alpha_norm_controls}
    For all $q\in \mathbb{C}$ and $0<\alpha\leq 1$, there exist constants $C_1,C_2>0$ such that for all $A\in GL(d,\mathbb{R})$, we have
    \[ \sup_{x\ne y} \frac{ \left| e^{q \ln\|Ax\|}-e^{q\ln\|Ay\|} \right|}{\|x-y\|^\alpha} \leq C_1 \eta^{( (1+\alpha) |Re(q)|+2\alpha )} . \]
    
\end{lemma}
\vspace{0.1in}

% The following lemma connects the moment Lyapunov exponent and the Lyapunov exponent.

% \begin{lemma}\label{lem:momentderivatives}
% 	The first derivative of the moment Lyapunov function at $q=0$ is equal to the Lyapunov exponent
% 	\begin{align}\label{eq:Lambda'(0)}
% 		\Lambda' (0) &= \lim_{n\to\infty} \frac{1}{n} \mathbb{E} \left[ \ln  \left( |A_\omega^n\theta|  \right) \right]  = \lambda.
% 	\end{align}

	% The second derivative at $p=0$ gives the variance
 %    \begin{align}\label{eq:Lambda''(0)}
	% \Lambda''(0) &= \lim_{n\to\infty} \frac{1}{n} \mathbb{E} \left[ \left( \ln  \left( D T_\omega^n (x) \right)  \right)^2 \right] - \Big(\mathbb{E} \left[ \ln  \left( DT_\omega^n (x)  \right) \right] \Big)^2  \nonumber
	% 	\\
	% 	&=  \lim_{n\to\infty} \frac{1}{n}  \mathbb{E} \Big[  \ln  \left( D T_\omega^n (x)  \right)   - \mathbb{E} \left[ \ln  \left( D T_\omega^n (x)  \right) \right] \Big]^2.
    % \end{align}
%\end{lemma}

% \begin{proof}
%     This follows from the fact that $\Lambda(0) =0$, $\Lambda(q) \geq \lambda q$ and the analyticity of $\Lambda$.
%     For the complete argument, see  \cite{ArnoldStabFormula,StayHomble}.
% \end{proof}

% We also want to show that the dominant eigenfunctions of the Koopman operator are uniformly bounded as we vary the probability measure $\mu_\kappa$.

%We want to use the analytic perturbation theory also in the "direction" of $\kappa$. 

In order to prove analyticity of the eigenvalues and eigenfunctions, we begin by showing that the operator $\mathcal{P}_{q,\mu}$ is itself jointly analytic in $q$ and $\mu$. To the best of our knowledge, this result is an original perspective on questions of analyticity of the Lyapunov exponent. Following this, we prove a transfer result (Proposition \ref{EigenAnalytic}) inspired by Kato's perturbation theory (\cite{MR203473}) that yields the desired result in our context. This transfer result relies on the operator $\mathcal{P}_{q,\mu}$ exhibiting a spectral gap as stated in Theorem \ref{Spectral_Gap_Operator}.\\

We recall the definition of Fréchet analyticity (holomorphicity) in the context of Banach spaces following \cite{chae2020holomorphy}. It generalizes the existence of a power series expansion on $\mathbb{C}$.
In the following, let $E$ and $F$ be complex Banach spaces and $U\subset E$ an open set.

\begin{definition}
     The map $T: U\to F$ is said to be \textit{Fréchet holomorphic} (or Fréchet analytic) at $x_0\in U$ if there exist $r>0$ and a sequence of continuous, symmetric $n$-linear maps $T_n: E^n\to F$ such that 
    \[ T(x) = \sum_{n=0}^\infty T_n (x-x_0,\dots,x-x_0) \]
    for all $x\in B(x_0,r)$, where the series above converges in the norm of $F$. $f$ is said to be Fréchet holomorphic on $U$ if it is Fréchet holomorphic at every $x_0\in U$.
\end{definition}
We typically refer to Fréchet holomorphicity when simply writing holomorphic or analytic in this paper. A weaker form of analyticity is Gâteaux analyticity which can be thought of as analyticity along lines:

\begin{definition}
    A map $T: U \to F$ is said to be \textit{Gâteaux holomorphic} (or Gâteaux analytic) if for every $x_0\in U$ and $v\in E$, the map \[z\mapsto T(x_0+zv)\]
    is complx holomorphic on $\tilde{U}:= \{z\in \mathbb{C}: x_0+zv\in U\}$.
\end{definition}

In practice, a linear operator $T_x$ on a function space $X$ is Fréchet (Gâteaux) holomorphic in $x$ at $x_0$ if and only if the map $x\mapsto T_xf$ is Fréchet (Gâteaux) holomorphic at $x_0$ for all $f\in X$ and there exist $c>0$ and an open neighbourhood $W$ containing $x_0$ such that \[ \sup_{x\in W} \|T_x\| <c \]
(local boundedness). We show local boundedness for our operator later, so it is sufficient to establish analyticity of $\mathcal{P}_{q,\mu}f$ for any $f\in \mathcal{C}^\alpha(\mathbb{S}^{d-1})$.
The following result allows us to verify Fréchet-analyticity through simpler criteria:
\begin{proposition}{\cite[Theorem 8.7]{mujica2010complex}}  \label{GateauToFrechet}
$T$ is Fréchet holormorphic on $U$ if and only if $T$ is Gâteaux holomorphic and continuous on $U$.
\end{proposition}

\section{Analyticity of the Twisted Koopman Operator}\label{Analyticity_of_Koopman_Section}
An important step towards the proof of Theorem \ref{MainTheorem1Analyticity} is showing that the operator $\mathcal{P}_{q,\mu}$ is jointly Fréchet-analytic in both $q$ and $\mu$.

Throughout this section, we implicitly extend the domain $\mu$ to the complex setting (the space of probability measures on $S$ is extended to the space of complex finite measures on $M(S)$) in order to use functional analysis tools to establish holomorphicity. It then follows that the restriction to probability measures is real-analytic. See for instance \cite[Section 3.2]{amorim2025analiticity} for a similar approach to analyticity of $\lambda_\mu$.

\begin{proposition}\label{Joint_Analyticity}
    There exists $q_0>0$ and an open subset $U\subset M(S)$ containing $\mu_0$ such that the map $(q,\mu)\mapsto \mathcal{P}_{q,\mu}$ is Fréchet-analytic on $B(0,q_0)\times U$.
\end{proposition}

The proof of Proposition \ref{Joint_Analyticity} is split into three steps.
First, we recall in Lemma \ref{Analytic_in_q} that $\mathcal{P}_{q,\mu}$ is analytic in $q$ in a neighbourhood of the origin, a known result from \cite{RM_Book}. We then proceed to show that the operator is analytic in $\mu$ in a neighbourhood of $\mu_0$ via Proposition \ref{GateauToFrechet}. Finally, a result from functional analysis allows us to conclude joint analyticity in $q$ and $\mu$.

\begin{lemma}\cite[Proposition 4.1]{RM_Book}\label{Analytic_in_q}
    Let $\mu$ be a probability measure supported on a subset of a compact set $S\subset GL(d,\mathbb{R})$. For all $\alpha>0$, there exists $q_0>0$ such that the map $T: \mathbb{C}\to L(\mathcal{C}^{\alpha}(\mathbb{S}^{d-1}))$ given by $T(q)=\mathcal{P}_{q,\mu}$ is analytic on $B(0,q_0)$.
\end{lemma}
Note that the result in \cite[Proposition 4.1]{RM_Book} is more general: it establishes analyticity when the measure $\mu$ has finite exponential moments. In our context, since the matrix norms are uniformly bounded, we obtain analyticity for any $\alpha>0$.

We now move to the second step: proving analyticity in the measure $\mu$. The following lemma establishes continuity of $\mu \mapsto \mathcal{P}_{q,\mu}$ on $M(S)$, the first ingredient of Holomorphicity according to Proposition \ref{GateauToFrechet}.

\begin{lemma}\label{Bounded_in_mu}
    For any $q\in \mathbb{C}$, the map $T: M(S) \to L(C^\alpha (\mathbb{S}^{d-1}))$ given by $\mu \mapsto \mathcal{P}_{q,\mu}$ is continuous.
\end{lemma}
\begin{proof} The integrand of the operator is denoted by $K_{q,A}f(x) = e^{q \ln\|Ax\|} f(Ax/\|Ax\|)$ for $A\in S, q\in \mathbb{C}, f\in\mathcal{C}^\alpha(\mathbb{S}^{d-1}) $. We have
    \[ \left| K_{q,A}f(x) \right| \leq \eta^{|Re(q)|} \Vert f \Vert_\infty . \]
    Similarly, by Lemma \ref{alpha_norm_controls}, for any $x\ne y\in \mathbb{S}^{d-1}$, we have
   \begin{align*} \left| \frac{K_{q,A}f(x)-K_{q,A}f(y)}{\Vert x-y \Vert^\alpha} \right| & \leq \left| e^{z \ln \|Ax\|} \right| \frac{\left| f(Ax/\|Ax\|)-f(Ay/\|Ay\|)  \right|}{\|x-y\|^\alpha} \\ & + \left| f(Ay/\|Ay\|) \right| \frac{ \left| e^{q \ln\|Ax\|}-e^{q\ln\|Ay\|} \right|}{\|x-y\|^\alpha}
   \\ &\leq [f]_\alpha  \eta^{|Re(q)|} + c \Vert f \Vert_{\infty}  \eta^{( (1+\alpha) |Re(q)|+2 )}  \\
   & \leq C ([f]_\alpha+||f||_\infty) = C \|f\|_{\mathcal{C}^\alpha}.
   \end{align*}
for some $C>0$ independent of $A\in S$. This means that 
\[ \sup_{A\in S} \|K_{q,A}\|_{L(\mathcal{C}^\alpha)}\leq C(q,\eta). \]

Let $\mu_1,\mu_2\in M(S)$. For any $f\in \mathcal{C}^\alpha (\mathbb{S}^{d-1}),$ we have 
\[ \mathcal{P}_{q,\mu_1}f- \mathcal{P}_{q,\mu_2}f = \int_S K_{q,A} f \ (\mu_1 - \mu_2)d(A). \]
Taking the $\mathcal{C}^\alpha$ norm, we obtain
\[ \| \mathcal{P}_{q,\mu_1}f- \mathcal{P}_{q,\mu_2}f  \|_{\mathcal{C}^\alpha} \leq \int_S  \| K_{q,A} f \|_{C^\alpha} \ (\mu_1 - \mu_2)d(A) \leq \sup_{A\in S} \| K_{q,A} f \|_{C^\alpha} \|\mu_1-\mu_2\|_{TV}  \]
and hence, 
\[ \| \mathcal{P}_{q,\mu_1}- \mathcal{P}_{q,\mu_2} \|_{L(\mathcal{C}^\alpha)} \leq C(q,\eta) \|\mu_1-\mu_2\|_{TV} \]
which proves continuity on $M(S)$.

\end{proof}

We can now prove analyticity in the measure, and then joint analyticity in $q$ and $\mu$.

\begin{proof}[Proof of Proposition \ref{Joint_Analyticity}]

Let $B\subset M(S)$ be an open neighbourhood and $f\in C^\alpha(\mathbb{S}^{d-1})$. In order to prove that the map $T_f: B\to C^\alpha(\mathbb{S}^{d-1})$ given by $T_f(\mu) = \mathcal{P}_{q,\mu}f$ is Gâteaux holomorphic, we need to show that for every $\mu\in B$ and $\nu\in M_0(S)$, the map $z\mapsto T(\mu+z\nu)$ is holomorphic on $\{z\in \mathbb{C}: \mu + z\nu \in B\}$. It is thus immediately clear that $T_f$ is Gâteaux holomorphic on any open neighbourhood of $M(S)$ as the map described above is affine in $z$ by linearity of integrals. Since this is true for any function $f$, the operator $P_{q,\mu}$ is Gâteaux holomorphic in $\mu$ on $B$.

Now, by Proposition \ref{GateauToFrechet}, $T$ is holomorphic on an open subset $U$ of a Banach space if and only if $T$ is Gâteaux holomorphic and continuous on $U$. Hence, combining the above with Lemma \ref{Bounded_in_mu}, we obtain that $\mu\mapsto \mathcal{P}_{q,\mu}$ is holomorphic on any open subset of $M(S)$.\\

We have now established that the Koopman operator is separately analytic in both $q$ (Lemma \ref{Analytic_in_q}) and $\mu$. It remains to prove that it is jointly analytic in both. From the proof of Lemma \ref{Bounded_in_mu}, it is straightforward to see that the Koopman operator is locally bounded in both $q$ and $\mu$ in the sense that for any open ball $B\subset M(S)\times \mathbb{C}$, we have 
\[ \sup_{(q,\mu)\in B} \Vert \mathcal{P}_{q,\mu}\Vert_{L(C^\alpha)} <\infty. \]
By \cite[Lemma 8.10]{mujica2010complex}, for any Banach spaces $E_1,E_2$ and $F$ and for $A$ an open subset of $E_1\times E_2$, if a map $f:A\to F$ is
separately holomorphic and locally bounded, then $f$ is holomorphic on $A$.
This proves that the map $(q,\mu)\mapsto P_{q,\mu}$ is analytic on $U\times B(0,q_0)$ for any open set $U\subset M(S)$.

\end{proof}

\section{Proof of Theorem \ref{MainTheorem1Analyticity}}\label{Proof_Of_Main_Section}
We now require a result to transfer Fréchet analyticity from the operator to its leading eigenvalue $e^{\Lambda_\mu(q)}$ and its normalized eigenfunctions. The proposition below is a variation of \cite[Theorem VII.1.7]{MR203473} from Kato's \textit{Perturbation Theory for Linear Operators}: we restrict the setting to the existence of a simple dominant eigenvalue, but we allow the index of the operator to be an element of a complex Banach space rather than $\mathbb{C}^n$. For a complex Banach space $X$, endow the space of linear operators $\mathcal{L}(X)$ from $X$ to itself with the operator norm induced from $X$.

\begin{proposition}\label{EigenAnalytic}
    Let $X$ be a complex Banach space, $U\subset X$ an open set, and $Y$ a Banach space. Suppose that $T: U \to \mathcal{L}(Y)$ (endowed with the operator norm) is Fréchet-holomorphic. Fix $x_0\in X$ such that the operator $T(x_0)$ has a simple isolated eigenvalue $l_0$. Then, there exists a open neighbourhood $V\subset X$ of $x_0$ and holomorphic maps $l: V \to \mathbb{C}$ and $h: V\to Y\setminus\{0\}$ such that $l(0)=l_0$ and \[ T(x)h(x)=l(x)h(x) \quad \text{for all } x\in V.  \] 
\end{proposition}

\begin{proof}
    For any operator $T$, denote by $\sigma(T)$ and $\rho(T)$ the spectrum and resolvent set of $T$ respectively. Since $l_0$ is an isolated eigenvalue, we have $\delta=d(l_0, \sigma(T(x_0))\setminus\{l_0\})>0$. Pick $r<\delta/2$ and define the closed curve \[\Gamma = \{\xi\in \mathbb{C}: |\xi-l_0|=r\}.\]
    It is clear that $\Gamma\subset\rho(T(x_0))$. Consider the resolvent operator $\tilde{T}_\xi(x)=(\xi I - T(x))^{-1}\in \mathcal{L}(Y)$ for $\xi\in\Gamma$. It is known that the map $\xi\mapsto \tilde{T}_\xi(x_0)$ is holomorphic on $\rho(T(x_0))$. This means the map is also continuous, and since the curve $\Gamma$ is compact, we have 
    \[ M_{\Gamma} = \sup_{\xi \in \Gamma} \Vert \tilde{T}_\xi(x_0) \Vert<\infty. \]
    By continuity of $x\mapsto T(x)$, there exists an open neighbourhood $V\subset U$ of $x_0$ such that for all $x\in V$, $\Vert T(x)-T(x_0)\Vert < \frac{1}{2 M_\Gamma}$. Define the operator $K: \Gamma \times V \to \mathcal{L}(Y)$ by \[ K(\xi,x) = \tilde{T}_\xi(x_0) \ (T(x)-T(x_0)) \]
    for all $x\in V, \xi\in\Gamma$.
    Then we can write $\xi I - T(x) = (\xi I - T(x_0)) (I-K(\xi,x))$. Notice that \[ \Vert K(\xi,x)\Vert \leq \Vert \tilde{T}_\xi(x_0) \Vert \ \Vert T(x)-T(x_0)\Vert < M_\Gamma ( 2 M_\Gamma)^{-1}=\frac{1}{2} . \]
    This implies that $I-K(\xi,x)$ is invertible with $(I-K)^{-1} = \sum_n K^n$ converging uniformly in $\xi\in \Gamma$. Hence, $(\xi I - T(x))$ is invertible and $\Gamma\subset \rho(T(x))$ for all $x\in V$. Practically this means that $\Gamma$ isolates a part of the spectrum in its interior for all the operators $T(x)$ with $x$ close enough to $x_0$. We also obtain a uniform bound in $x$ and $\xi$ on the resolvent along the curve:
    \begin{equation}\label{NormOfResolvent}
        \Vert \tilde{T}_\xi(x) \Vert\leq \frac{1}{1-\Vert K(\xi,x) \Vert} \Vert \tilde{T}_\xi(x_0)\Vert \leq 2M_\Gamma
    \end{equation}   
    Now consider the Riesz projection associated with the part of the spectrum of $T(x)$ the lies within the interior of the curve $\Gamma$: it is defined by the integral
    \[ P(x)=\frac{1}{2\pi i} \int_\Gamma \tilde{T}_\xi(x) d\xi \quad \text{for all } x\in V. \]
    For every $x\in V$, $P(x)$ is a linear operator whose range is the direct sum of the generalised eigenspaces corresponding to the eigenvalues of $T(x)$ that lie inside of $\Gamma$ (\cite{MR203473}).\\

    We claim the map $P: V \to \mathcal{L}(Y)$ if Fréchet-holomorphic. We prove this using Proposition \ref{GateauToFrechet}: we show the operator is continuous on $V$ and we establish Gâteaux holomorphicity, the combination of which proves the claim.  We start by proving holomorphicity along lines (Gâteaux holomorphicity), i.e. for any $h\in E$, the map $z\mapsto P(x+zh)$ is analytic for $z\in\mathbb{C}$ small enough. Since $x\mapsto T(x)$ is Fréchet holomorphic on $V$, the map $x\mapsto \tilde{T}_\xi(x)$ is also Fréchet holomorphic by \cite[Theorem VII.1.3]{MR203473} hence Gâteaux holomorphic. This means that $z\mapsto \tilde{T}_\xi(x+zh)$ is holomorphic on $B(0,R)$, where $R$ is such that $B(0,R)\subset \{z\in \mathbb{C}: x+zh \in V\}  $. By \cite[Corollary 7.2]{mujica2010complex} there exists a sequence $(c_n(\xi))_{n\in \mathbb{N}} \subset \mathcal{L}(Y)$ such that \[ \tilde{T}_\xi(x+zh) = \sum_{n=0}^\infty z^n c_n(\xi) \]
    with the series converging absolutely in operator norm for $|z|\leq s$ for $0<s<R$, and \[ c_n(\xi) = \frac{1}{2\pi i} \int_{|u|=R} \frac{\tilde{T}_\xi(x+zh)}{u^{n+1}}du. \]

    We can thus integrate termwise: $P(x+zh)=\sum_n z^n B_n$ where 
    \[ B_n = \frac{1}{2\pi i} \int_\Gamma c_n(\xi) d\xi \in \mathcal{L}(Y).  \]
    This series also converges absolutely in operator norm since, by \eqref{NormOfResolvent}, we have

\begin{align*}
    \sum_{n=0}^\infty \|B_n\| |z|^n & \leq \frac{\text{length}(\Gamma)}{2\pi} \sum_{n=0}^\infty |z|^n  \sup_{\xi\in\Gamma} \Vert c_n(\xi)\Vert \\ &  \leq \frac{r}{2\pi} \sum_{n=0}^\infty |z|^n  \sup_{\xi\in\Gamma}  \int_{|u|=R}   \frac{\Vert \tilde{T}_\xi(x+zh) \Vert}{u^{n+1}}du  \\ & \leq \frac{rM_\Gamma}{\pi} \sum_{n=0}^\infty \left(\frac{|z|}{R}\right) ^n <\infty
\end{align*}
since $|z|<R$.
    This proves that $P(x)$ is Gâteaux holomorphic. \\
    
    By \eqref{NormOfResolvent}, we have 
    \[ \sup_{x\in V} \Vert P(x)\Vert \leq \sup_{x\in V} \frac{1}{2\pi} \int_\Gamma \Vert \tilde{T}_\xi (x) \Vert d\xi  \leq \frac{\text{length}(\Gamma)}{2\pi} \sup_{x\in V, \xi\in\Gamma} \Vert \tilde{T}_\xi(x)\Vert   \leq  2r M_\Gamma <\infty \]
    which establishes boundedness in $x\in V$. By \cite[Proposition 8.6]{mujica2010complex}, a Gâteaux-holomorphic mapping is continuous if and only if it is locally bounded. Hence $P(x)$ is continuous in $x$.
    Together with Gâteaux holomorphicity we conclude that $P$ is Fréchet holomorphic on $V$.\\

    Since $l_0$ is a simple eigenvalue of $T(x_0)$ and $P$ projects onto the eigenspace corresponding to the eigenvalue $l_0$ operator, Rank$P(x_0)=1$. By continuity of $x\mapsto P(x)$, so is the rank of the operator, and since the rank can only take integer value, this means that Rank$P(x)=1$ on $x\in V$. This means that the eigenvalues $l(x)$ of $T(x)$ isolated within $\Gamma$ are also simple with a one-dimensional eigenspaces.

    In order to construct an eigenfunction, choose an arbitrary $u\in Y$ such that $P(x_0)u\ne 0$ and set $h(x)=P(x)u$. Since the eigenspace is one-dimensional, Range$P(x)=\mathbb{C}h(x)$ corresponds to the eigenspace of $l(x)$. Since $x\mapsto P(x)$ is holomorphic, so is $h(x)$ on $V$. We conclude all the eigenfunctions are analytic on $V$. For $h(x)\ne 0$, the eigenvalue can be expressed as \[l(x)=\frac{T(x)h(x)}{h(x)}\in \mathbb{C}.\]
    Hence the eigenvalue $l(x)$ is analytic in $x$ on $V$. This completes the proof.

\end{proof}

The following corollary finalises the proof of Theorem \ref{MainTheorem1Analyticity} (1)-(2).

\begin{corollary}\label{Corr_analyticity}
    Let $S\subset GL(d,\mathbb{R})$ be compact, and let $\mu_0\in P(S)$ such that the support of $\mu_0$ is strongly irreducible and proximal. There exists an open neighbourhood $V\subset \mathbb{C}\times M(S)$ of $(0,\mu_0)$  and an open neighbourhood $U\subset M(S)$ of $\mu_0$ such that the map $(q,\mu)\mapsto \Lambda_\mu(q) $ is analytic on $V$ and the map $\mu \mapsto \Lambda_\mu^{(k)}(0)$ is analytic on $U$ for all $k\in \mathbb{N}$. 
%     Moreover, there exist a constant $K>0$ and an open neighbourhood $U\subset M(S)$ of $\mu_0$ such that 
%     the normalised eigenfunction $\phi_{q,\mu}$ of $\mathcal{P}_{q,\mu}$ and the map $\psi_\mu=\frac{d\phi_{0,\mu}}{dq}$ satisfy
% \begin{equation}\label{eigenBounds1}
%     K^{-1} < \inf_{(q,\mu)\in V} \  \inf_{x\in \mathbb{S}^{d-1}} \phi_{q,\mu}(x) \leq \sup_{(q,\mu)\in V} \  \sup_{x\in \mathbb{S}^{d-1}} \phi_{q,\mu}(x) <K
% \end{equation}   
% and \begin{equation}\label{eigenBounds2}
%     \sup_{ \mu\in U} \  \sup_{x\in \mathbb{S}^{d-1}} |\psi_{\mu}(x)| <K.
% \end{equation} 
\end{corollary}
\begin{proof}
Set $X=M(S)\times \mathbb{C}$ (product of two Banach spaces is a Banach space under the direct product norm), $Y=C^\alpha(\mathbb{S}^{d-1})$, $x_0=(\mu_0,0)$ and $l_0 = e^{\Lambda_{\mu_0}(0)}$ which by Lemma \ref{Analytic_in_q} is a dominant simple eigenvlue of $P_{0,\mu_0}$. The operator is jointly Fréchet analytic in $\mu$ and $q$ in a neighbourhood of $(\mu_0,0)$ by Proposition \ref{Joint_Analyticity}. Hence, by Proposition \ref{EigenAnalytic}, there exists an open neighbourhood $V$ of $(\mu_0,0)$ such that the eigenvalue $l(\mu,q) = e^{\Lambda_\mu (q)}$ and its normalized eigenfunction $\phi_{q,\mu}$ are jointly analytic in $q$ and $\mu$ on $V$.\\

This implies that $\Lambda_\mu(q)$ is analytic in both variables on $V$. In particular, the map $\mu\mapsto \Lambda^{(k)}_\mu (0)$ is analytic on some open neighbourhood $U\subset M(S)$ of $\mu_0$ for all $k\in\mathbb{N}$. We thus recover the result by \cite{amorim2025analiticity} about $\lambda_{\mu} = \Lambda_\mu'(0)$ being analytic in the measure, in addition to the asymptotic variance $\sigma^2_\mu = \Lambda_\mu''(0)$ being analytic on $U$.
\end{proof}

We prove Theorem \ref{MainTheorem1Analyticity} (3) in the next corollary by applying the result of Proposition \ref{EigenAnalytic} to the dual operator $\mathcal{P}^*_{0,\mu}\in \mathcal{L}(\mathcal{C}^\alpha (\mathbb{S}^{d-1})^*)$. $\nu_\mu$ is the eigenvector of $\mathcal{P}_{0,\mu}^*$ at eigenvalue $e^{\Lambda_\mu(0)}=e^0=1$ representing the unique invariant measure on $\mathbb{S}^{d-1}$ of the random matrix product when $\mu$ is a strongly irreducible and proximal probability measure. In this context we have \[ \nu_\mu (f) = \left<f,\nu \right>  = \int_{\mathbb{S}^{d-1}} f(x) \ d\nu_\mu(x) .\]
\begin{corollary}\label{corr_statio}
    Let $S\subset GL(d,\mathbb{R})$ be compact, and let $\mu_0\in P(S)$ such that the support of $\mu_0$ is strongly irreducible and proximal. There exists an open neighbourhood $U\subset M(S)$ of $\mu_0$ such that the map $\mu \mapsto \nu_\mu$ is analytic on $U$.  
\end{corollary}
\begin{proof}
    We begin by showing that the map $T^*$ given by $T^*(\mu)= \mathcal{P}^\star_{0,\mu}$ is Fréchet-analytic on some open neighbourhood of $\mu_0$. The map $T$ given by $T(\mu)=\mathcal{P}_{0,\mu}$ is Fréchet analytic on some open neighbourhood $U\subset M(S)$ of $\mu_0$ by Proposition \ref{Joint_Analyticity}. The map $D: \mathcal{L}(\mathcal{C}^\alpha (\mathbb{S}^{d-1})) \to \mathcal{L}(\mathcal{C}^\alpha (\mathbb{S}^{d-1})^*)$ given by  $D(\mathcal{P})=\mathcal{P}^*$ is a linear isometry, in particular it is continuous. Since the composition of a continuous linear map and a Fréchet analytic map is analytic (see for example \cite[Exercise 5.A]{mujica2010complex}), the map $T^* = D\circ T$ is Fréchet analytic on $U$. 

    By \cite[Theorem III.6.22]{kifer1982perturbations}, if $1$ is a simple isolated eigenvalue of $\mathcal{P}_{0,\mu_0}$, $1$ is a simple isolated eigenvalue of $\mathcal{P}^*_{0,\mu_0}$. Set $X=M(S)$, $Y=\mathcal{C}^\alpha (\mathbb{S}^{d-1})^*$ endowed with the dual norm \[ \|\nu\|_{\mathcal{C}^\alpha(\mathbb{S}^{d-1})^*} = \sup\{ |\nu(f)| : ||f||_{\mathcal{C}^\alpha} \leq 1, f\in \mathcal{C}^\alpha (\mathbb{S}^{d-1}) \}   ,\] $x_0=\mu_0$ and $l_0=1$. By proposition \ref{EigenAnalytic}, there exists an open neighbourhood $U'\subset M(S)$ of $\mu_0$ such that the eigenvector $\nu_\mu$ is Fréchet analytic on $U'$.
\end{proof}

\subsection*{Acknowledgements}
The authors gratefully acknowledge useful discussions with Artur Amorim and Richard Aoun. 
CC, VPHG and JSWL have been supported by the EPSRC Centre for Doctoral Training in Mathematics of Random Systems: Analysis, Modelling and Simulation (EP/S023925/1). VPHG thanks the Mathematical Institute of Leiden University for their hospitality. 
JSWL has been supported by the EPSRC (EP/Y020669/1) and thanks JST (Moonshot R \& D Grant Number JPMJMS2021 - IRCN, University of Tokyo) and GUST (Kuwait) for their research support.

\bibliographystyle{plain} 
\bibliography{ERMP}

@article{ArnoldStabFormula,
	author = {Arnold, L.},
	doi = {10.1137/0144057},
	fjournal = {SIAM Journal on Applied Mathematics},
	issn = {0036-1399},
	journal = {SIAM J. Appl. Math.},
	mrclass = {34F05 (60H10 93E15)},
	mrreviewer = {S.\ M.\ Khrī sanov},
	number = {4},
	pages = {793–802},
	title = {A formula connecting sample and moment stability of linear stochastic systems},
	url = {https://doi.org/10.1137/0144057},
	volume = {44},
	year = {1984}
}

@book {RM_Book,
    AUTHOR = {Bougerol, Philippe and Lacroix, Jean},
     TITLE = {Products of random matrices with applications to {S}chr\"{o}dinger
              operators},
    SERIES = {Progress in Probability and Statistics},
    VOLUME = {8},
 PUBLISHER = {Birkh\"{a}user Boston, Inc., Boston, MA},
      YEAR = {1985},
     PAGES = {xii+283},
      ISBN = {0-8176-3324-3},
   MRCLASS = {60B15 (47F05 58G40 60H25 82A42)},
  MRNUMBER = {886674},
MRREVIEWER = {Shinichi Kotani},
       DOI = {10.1007/978-1-4684-9172-2},
       URL = {https://doi.org/10.1007/978-1-4684-9172-2},
}

@article{LDP1,
title = {Precise large deviation asymptotics for products of random matrices},
journal = {Stochastic Processes and their Applications},
volume = {130},
number = {9},
pages = {5213-5242},
year = {2020},
issn = {0304-4149},
doi = {https://doi.org/10.1016/j.spa.2020.03.005},
url = {https://www.sciencedirect.com/science/article/pii/S0304414919304090},
author = {Xiao, H. and Grama, I.  and Liu, Q. }
}

@article{furstenberg1960products,
    AUTHOR = {Furstenberg, H. and Kesten, H.},
     TITLE = {Products of random matrices},
   JOURNAL = {Ann. Math. Statist.},
  FJOURNAL = {Annals of Mathematical Statistics},
    VOLUME = {31},
      YEAR = {1960},
     PAGES = {457--469},
      ISSN = {0003-4851},
   MRCLASS = {60.00},
  MRNUMBER = {121828},
MRREVIEWER = {R. E. Bellman},
       DOI = {10.1214/aoms/1177705909},
       URL = {https://doi.org/10.1214/aoms/1177705909},
}

@article{guivarc1985frontiere,
  title={Frontiere de Furstenberg, propri{\'e}t{\'e}s de contraction et th{\'e}oremes de convergence},
  author={Guivarc'h, Y. and Raugi, A.},
  journal={Zeitschrift f{\"u}r Wahrscheinlichkeitstheorie und Verwandte Gebiete},
  volume={69},
  pages={187--242},
  year={1985},
  publisher={Springer}
}

@article{baxendale1998stability,
    AUTHOR = {Baxendale, P. H. and Khasminskii, R. Z.},
     TITLE = {Stability index for products of random transformations},
   JOURNAL = {Adv. in Appl. Probab.},
  FJOURNAL = {Advances in Applied Probability},
    VOLUME = {30},
      YEAR = {1998},
    NUMBER = {4},
     PAGES = {968--988},
      ISSN = {0001-8678},
   MRCLASS = {60H25 (60J05 92D25)},
  MRNUMBER = {1671091},
MRREVIEWER = {Andrej A. Dorogovtsev},
       DOI = {10.1239/aap/1035228203},
       URL = {https://doi.org/10.1239/aap/1035228203},
}

@article{avila2023continuity,
  title={Continuity of the Lyapunov exponents of random matrix products},
  author={Avila, A. and Eskin, A. and Viana, M.},
  year={2023},
  eprint={2305.06009},
  journal={arXiv preprint arXiv:2305.06009},
  archivePrefix={arXiv},
  primaryClass={math.DS},
  url={https://arxiv.org/abs/2305.06009},
}

@book {MR203473,
    AUTHOR = {Kato, T.},
     TITLE = {Perturbation theory for linear operators},
    SERIES = {Die Grundlehren der mathematischen Wissenschaften},
    VOLUME = {Band 132},
 PUBLISHER = {Springer-Verlag New York, Inc., New York},
      YEAR = {1966},
     PAGES = {xix+592},
   MRCLASS = {47.00 (47.48)},
  MRNUMBER = {203473},
MRREVIEWER = {L.\ de Branges},
}

@article {MR3945748,
    AUTHOR = {Sert, Cagri},
     TITLE = {Large deviation principle for random matrix products},
   JOURNAL = {Ann. Probab.},
  FJOURNAL = {The Annals of Probability},
    VOLUME = {47},
      YEAR = {2019},
    NUMBER = {3},
     PAGES = {1335--1377},
      ISSN = {0091-1798,2168-894X},
   MRCLASS = {60F10 (20P05 22E46 60B20)},
  MRNUMBER = {3945748},
MRREVIEWER = {Ofer\ Zeitouni},
       DOI = {10.1214/18-AOP1285},
       URL = {https://doi.org/10.1214/18-AOP1285},
}

@article{amorim2025analiticity,
  title={Analiticity of the Lyapunov exponents of random products of matrices},
  author={Amorim, A. and Dur{\~a}es, M. and Melo, A.},
  journal={arXiv preprint arXiv:2501.19286},
  year={2025},
  archivePrefix={arXiv},
  primaryClass={math.DS},
  url={https://arxiv.org/abs/2501.19286},
}

@book{chae2020holomorphy,
  title={Holomorphy and Calculus in Normed Spaces},
  author={Chae, S. B.},
  year={2020},
  publisher={CRC Press}
}

@book {mujica2010complex,
    AUTHOR = {Mujica, Jorge},
     TITLE = {Complex analysis in {B}anach spaces},
    SERIES = {North-Holland Mathematics Studies},
    VOLUME = {120},
      NOTE = {Holomorphic functions and domains of holomorphy in finite and
              infinite dimensions,
              Notas de Matem\'atica, 107. [Mathematical Notes]},
 PUBLISHER = {North-Holland Publishing Co., Amsterdam},
      YEAR = {1986},
     PAGES = {xii+434},
      ISBN = {0-444-87886-6},
   MRCLASS = {46G20 (32Dxx 32Exx 58C10)},
  MRNUMBER = {842435},
MRREVIEWER = {Martin\ Schottenloher},
}

@article{furstenbergkifer1983,
    AUTHOR = {Furstenberg, H. and Kifer, Y.},
     TITLE = {Random matrix products and measures on projective spaces},
   JOURNAL = {Israel J. Math.},
  FJOURNAL = {Israel Journal of Mathematics},
    VOLUME = {46},
      YEAR = {1983},
    NUMBER = {1-2},
     PAGES = {12--32},
      ISSN = {0021-2172},
   MRCLASS = {22D40 (28B99 60J15)},
  MRNUMBER = {727020},
MRREVIEWER = {Yves Guivarc'h},
       DOI = {10.1007/BF02760620},
       URL = {https://doi.org/10.1007/BF02760620},
}

@article{hennion1984loi,
  title={Loi des grands nombres et perturbations pour des produits r{\'e}ductibles de matrices al{\'e}atoires ind{\'e}pendantes},
  author={Hennion, H.},
  journal={Zeitschrift f{\"u}r Wahrscheinlichkeitstheorie und Verwandte Gebiete},
  volume={67},
  number={3},
  pages={265--278},
  year={1984},
  publisher={Springer}
}

@article{kifer1982perturbations,
  title={Perturbations of random matrix products},
  author={Kifer, Y.},
  journal={Zeitschrift f{\"u}r Wahrscheinlichkeitstheorie und Verwandte Gebiete},
  volume={61},
  number={1},
  pages={83--95},
  year={1982},
  publisher={Springer-Verlag Berlin/Heidelberg}
}

@inproceedings{lepage1989,
  title={R{\'e}gularit{\'e} du plus grand exposant caract{\'e}ristique des produits de matrices al{\'e}atoires ind{\'e}pendantes et applications},
  author={Le Page, {\'E}.},
  booktitle={Annales de l'IHP Probabilit{\'e}s et statistiques},
  volume={25},
  number={2},
  pages={109--142},
  year={1989}
}

@incollection {peres2006analytic,
    AUTHOR = {Peres, Y.},
     TITLE = {Analytic dependence of {L}yapunov exponents on transition
              probabilities},
 BOOKTITLE = {Lyapunov exponents ({O}berwolfach, 1990)},
    SERIES = {Lecture Notes in Math.},
    VOLUME = {1486},
     PAGES = {64--80},
 PUBLISHER = {Springer, Berlin},
      YEAR = {1991},
      ISBN = {3-540-54662-6},
   MRCLASS = {60J15 (60B15)},
  MRNUMBER = {1178947},
MRREVIEWER = {Albert\ Raugi},
       DOI = {10.1007/BFb0086658},
       URL = {https://doi.org/10.1007/BFb0086658},
}

@article {MR1158741,
    AUTHOR = {Peres, Yuval},
     TITLE = {Domains of analytic continuation for the top {L}yapunov
              exponent},
   JOURNAL = {Ann. Inst. H. Poincar\'e{} Probab. Statist.},
  FJOURNAL = {Annales de l'Institut Henri Poincar\'e. Probabilit\'es et
              Statistiques},
    VOLUME = {28},
      YEAR = {1992},
    NUMBER = {1},
     PAGES = {131--148},
      ISSN = {0246-0203},
   MRCLASS = {60B99 (47N30 58F11)},
  MRNUMBER = {1158741},
MRREVIEWER = {Philippe\ Bougerol},
       URL = {http://www.numdam.org/item?id=AIHPB_1992__28_1_131_0},
}

@article {MR534172,
    AUTHOR = {Ruelle, D.},
     TITLE = {Analycity properties of the characteristic exponents of random
              matrix products},
   JOURNAL = {Adv. in Math.},
  FJOURNAL = {Advances in Mathematics},
    VOLUME = {32},
      YEAR = {1979},
    NUMBER = {1},
     PAGES = {68--80},
      ISSN = {0001-8708},
   MRCLASS = {58F15 (28D99 82A15)},
  MRNUMBER = {534172},
MRREVIEWER = {Michael\ Keane},
       DOI = {10.1016/0001-8708(79)90029-X},
       URL = {https://doi.org/10.1016/0001-8708(79)90029-X},
}

@article {AounGuivarc,
    AUTHOR = {Aoun, Richard and Guivarc'h, Yves},
     TITLE = {Random matrix products when the top {L}yapunov exponent is
              simple},
   JOURNAL = {J. Eur. Math. Soc. (JEMS)},
  FJOURNAL = {Journal of the European Mathematical Society (JEMS)},
    VOLUME = {22},
      YEAR = {2020},
    NUMBER = {7},
     PAGES = {2135--2182},
      ISSN = {1435-9855,1435-9863},
   MRCLASS = {37H15 (20P05 37D25 60B15 60B20 60F10)},
  MRNUMBER = {4107504},
MRREVIEWER = {Zhiming\ Li},
       DOI = {10.4171/JEMS/962},
       URL = {https://doi.org/10.4171/JEMS/962},
}

@article {aounsert1,
    AUTHOR = {Aoun, Richard and Sert, Cagri},
     TITLE = {Stationary probability measures on projective spaces 1:
              block-{L}yapunov dominated systems},
   JOURNAL = {Math. Ann.},
  FJOURNAL = {Mathematische Annalen},
    VOLUME = {388},
      YEAR = {2024},
    NUMBER = {3},
     PAGES = {2573--2610},
      ISSN = {0025-5831,1432-1807},
   MRCLASS = {37H15 (37A20 60B15 60J05)},
  MRNUMBER = {4705746},
MRREVIEWER = {Paulo\ R. C. Ruffino},
       DOI = {10.1007/s00208-023-02585-y},
       URL = {https://doi.org/10.1007/s00208-023-02585-y},
}

@article {aounsert2,
    AUTHOR = {Aoun, Richard and Sert, Cagri},
     TITLE = {Stationary probability measures on projective spaces 2: the
              critical case},
   JOURNAL = {Duke Math. J.},
  FJOURNAL = {Duke Mathematical Journal},
    VOLUME = {174},
      YEAR = {2025},
    NUMBER = {7},
     PAGES = {1407--1430},
      ISSN = {0012-7094,1547-7398},
   MRCLASS = {37H15 (60B15 60J05)},
  MRNUMBER = {4912978},
MRREVIEWER = {Paulo\ R. C. Ruffino},
       DOI = {10.1215/00127094-2024-0056},
       URL = {https://doi.org/10.1215/00127094-2024-0056},
}

\end{document}